\documentclass[12pt]{article}
\usepackage[top=3.2cm,bottom=3.2cm,left=2.5cm,right=2.5cm]{geometry}
\usepackage{amssymb}
\usepackage{amsmath,amsthm}
\usepackage[latin1]{inputenc}
\usepackage[dvips]{graphicx}
\usepackage{hyperref}
\usepackage{color}
\usepackage{enumerate}
\usepackage{enumitem}
\usepackage{tikz}
\usepackage{xifthen}
\usepackage{verbatim}
\hypersetup{colorlinks=true, linkcolor=blue, citecolor=blue,urlcolor=blue}

\newtheorem{remark}{Remark}[section]

\newtheorem{lemma}[remark]{Lemma}
\newtheorem{theorem}[remark]{Theorem}
\newtheorem{proposition}[remark]{Proposition}

\newtheorem{corollary}[remark]{Corollary}

\title{Independent domination in central graphs}
\author{Abel Cabrera-Mart\'inez$^{1}$, Jos\'e Luis L\'opez-Carmona$^{1}$\\[5pt]
Ismael Rios-Villamar$^{2}$, Alejandro Serrano-D\'iaz$^{1}$\\[15pt]
{\small $^{1}$ Universidad de C\'ordoba, Departamento de Matem\'aticas, Campus de Rabanales, 14071,} \\
{\small C\'ordoba, Spain (acmartinez@uco.es, 2locaj@uco.es, aserrano1@uco.es)}\\[7pt]
{\small $^{2}$ Universidad Aut\'onoma de Guerrero,  Facultad de  Matem\'{a}ticas, Carlos E. Adame 54,}\\
{\small  Col. La Garita 39650, Acapulco, Guerrero, Mexico (18305783@uagro.mx)}\\
}

\date{ }
\begin{document}
\maketitle

\begin{abstract}
Let $G$ be a graph with vertex set $V(G)$. A set $I\subseteq V(G)$ is an independent dominating set of $G$ if no two vertices in $I$ are adjacent and every vertex in $V(G)\setminus I$ is adjacent to at least one vertex in $I$. The independent domination number of $G$ is the minimum cardinality among all independent dominating sets of $G$.  The aim of this article is to obtain tight bounds and closed formulas for the independent domination number of central graphs. The results are expressed in terms of parameters of the original graph from which the central graph is constructed. 
\end{abstract}


\noindent
{\it Math. Subj. Class. (2020)}: 05C69, 05C76.

\section{Introduction}

Let $G$ be a graph with vertex set $V(G)$ of order $n=|V(G)|$ and edge set $E(G)$ of size $m=|E(G)|$. Let $V(G)=\{v_1, \ldots , v_n\}$ and $V_E(G)=\{v^{i,j} : v_iv_j\in E(G)\}$.
The \emph{central graph} of a nontrivial connected graph $G$, denoted by $\mathtt{C}(G)$, is the graph with vertex set  $V(\mathtt{C}(G))=V(G)\cup V_E(G)$ and edge set
$E(\mathtt{C}(G))=\{v_iv^{i,j},v_jv^{i,j}: v^{i,j}\in V_E(G)\}\cup \{v_iv_j: v_iv_j\notin E(G)\}$.
In Figure~\ref{fig-central} we show a graph $G$ and its corresponding central graph~$\mathtt{C}(G)$.
Given a graph $G$, we denote the central graph of $\mathtt{C}(G)$ by $\mathtt{C}^2(G)$ instead of $\mathtt{C}(\mathtt{C}(G))$.

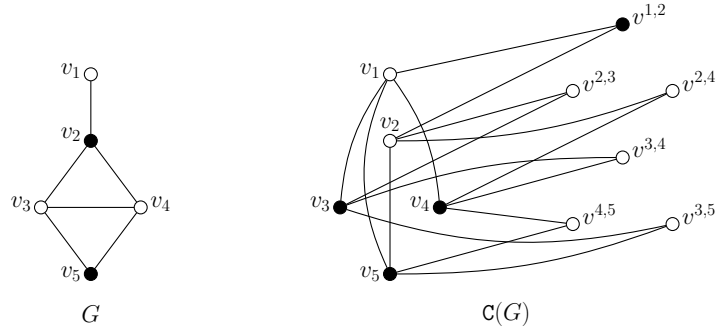
\begin{figure}[ht]
\centering
\begin{tikzpicture}[scale=.44, transform shape]
	
\node [draw, shape=circle] (a1) at  (6.5,-5.5) {};
\node at (5.9,-5.4) {\LARGE $v_1$};
\node [draw, shape=circle, fill=black] (a2) at  (6.5,-7.5) {};
\node at (5.9,-7.3) {\LARGE $v_2$};
\node [draw, shape=circle] (a3) at  (5,-9.5) {};
\node at (4.4,-9.4) {\LARGE $v_3$};
\node [draw, shape=circle] (a4) at  (8,-9.5) {};
\node at (8.6,-9.4) {\LARGE $v_4$};
\node [draw, shape=circle, fill=black] (a5) at  (6.5,-11.5) {};		
\node at (5.9,-11.5) {\LARGE $v_5$};

\node at (6.5,-12.7) {\LARGE $G$};

\draw(a1)--(a2)--(a3)--(a4)--(a2);
\draw(a3)--(a5)--(a4);

\node [draw, shape=circle] (ccc1) at  (15.5,-5.5) {};
\node at (14.9,-5.4) {\LARGE $v_1$};
\node [draw, shape=circle] (ccc2) at  (15.5,-7.5) {};
\node at (15.5,-7) {\LARGE $v_2$};
\node [draw, shape=circle, fill=black] (ccc3) at  (14,-9.5) {};
\node at (13.4,-9.4) {\LARGE $v_3$};
\node [draw, shape=circle, fill=black] (ccc4) at  (17,-9.5) {};
\node at (16.4,-9.4) {\LARGE $v_4$};
\node [draw, shape=circle, fill=black] (ccc5) at  (15.5,-11.5) {};
\node at (14.9,-11.5) {\LARGE $v_5$};

\node [draw, shape=circle, fill=black] (ddd1) at  (22.5,-4) {};
\node at (23.3,-3.7) {\LARGE $v^{1,2}$};
\node [draw, shape=circle] (ddd2) at  (21,-6) {};
\node at (21.8,-5.7) {\LARGE $v^{2,3}$};
\node [draw, shape=circle] (ddd3) at  (22.5,-8) {};
\node at (23.3,-7.7) {\LARGE $v^{3,4}$};
\node [draw, shape=circle] (ddd4) at  (24,-6) {};
\node at (24.8,-5.7) {\LARGE $v^{2,4}$};

\node [draw, shape=circle] (ddd5) at  (21,-10) {};
\node at (21.8,-9.7) {\LARGE $v^{4,5}$};
\node [draw, shape=circle] (ddd6) at  (24,-10) {};
\node at (24.8,-9.7) {\LARGE $v^{3,5}$};

\node at (19,-12.7) {\LARGE $\mathtt{C}(G)$};

\draw (ccc3) to[bend left=15] (ccc1);
\draw (ccc4) to[bend right=15] (ccc1);
\draw (ccc5) to[bend left=25] (ccc1);
\draw(ccc5)--(ccc2); 

\draw (ccc1)--(ddd1)--(ccc2)--(ddd2)--(ccc3);
\draw (ccc2) to[bend right=10] (ddd4);
\draw (ccc3) to[bend left=10] (ddd3);
\draw (ddd4)--(ccc4)--(ddd3);
\draw(ccc5)--(ddd5)--(ccc4);	
\draw (ccc5) to[bend right=10] (ddd6);	
\draw (ccc3) to[bend right=15] (ddd6);
\end{tikzpicture}
\caption{A graph $G$, and the corresponding central graph $\mathtt{C}(G)$.}\label{fig-central}
\end{figure} 

\vspace{.1cm}

\noindent
The central graph $\mathtt{C}(G)$ of a graph $G$ was introduced in \cite{Vernold} and has since attracted the attention of numerous researchers. One of the main problems in the study of $\mathtt{C}(G)$ is to determine exact values or tight bounds for certain parameters and to express them in terms of known invariants of $G$. In particular, we cite the following works on domination
theory in this graph operator: domination \cite{BFKB2026}, total domination \cite{KM2019}, double domination \cite{CRS2026} and double total domination \cite{RCRS2026}. In this article, we focus on the independent domination in central graphs.

\vspace{.1cm}

\noindent
A set of vertices $I\subseteq V(G)$ of a connected graph $G$ is called an \emph{independent dominating set} of $G$ if no two vertices in $I$ are adjacent and every vertex in $V(G)\setminus I$ is adjacent to at least one vertex in $I$. Let $\mathcal{ID}(G)$ be the set of independent dominating sets of $G$. The \emph{independent domination number} of $G$, denoted by $i(G)$, is the minimum cardinality among all sets in $\mathcal{ID}(G)$, that is, $i(G)=\min\{|I|:\, I\in \mathcal{ID}(G)\}.$ 
An $i(G)$-\emph{set} is defined as a set $I\in \mathcal{ID}(G)$ of cardinality $|I|=i(G)$. For each of the graphs given in Figure~\ref{fig-central}, the set of black vertices describes an independent dominating set of minimum cardinality. This concept and its notation were first introduced by Cockayne and Hedetniemi \cite{ind-1,ind-2}. Recent results on independent domination can be found in \cite{IndDom-2022,IndDom-2019,Cabrera2023,CR2024,CCP2023,IndDom-2017,IndDom-2022-1}.

\vspace{.1cm}

\noindent
The aim of this article is to obtain tight bounds and closed formulas for the independent domination number of central graphs. The results are expressed in terms of parameters of the original graph from which the central graph is constructed. 

\subsection{Notation and terminology}
Throughout the article, we use the following notation and terminology. Let $G$ be a finite, undirected, and simple graph. For each vertex $v_i\in V(G)$, let $N_G(v_i)$ denote the \emph{open neighborhood} of $v_i$ in $G$. The \emph{minimum degree} of $G$ is defined as $\delta(G)=\min\{|N_G(v_i)|: v_i\in V(G)\}$, and the \emph{maximum degree} of $G$ as $\Delta(G)=\max\{|N_G(v_i)|: v_i\in V(G)\}$. We denote
by $\overline{G}$ the complement of $G$. For any subset $X\subseteq V(G)$, let $G-X$ denote the graph obtained from $G$ by deleting all vertices in $X$ together with their incident edges. 

\vspace{.1cm}

\noindent
For two sets $X,Y\subseteq V(G)$, let $E_G(X,Y)$ denote the set of edges in $G$ joining a vertex of $X$ to a vertex of $Y$. For any subset $X\subseteq V(G)$, define $P_G(X) = \{v_i \in V(G)\setminus X:\ N_G(v_i) = X\}$, and we further define the function $\delta(P_G(X))$ by 
$$\delta(P_{G}(X)) = 
\begin{cases}
1 & \text{if } P_{G}(X) \neq \emptyset, \\[4pt]
0 & \text{otherwise.}
\end{cases}$$	

\noindent
The \emph{independence number} of $G$, denoted by $\alpha(G)$, is the maximum cardinality among all sets in $\mathcal{ID}(G)$, that is, $\alpha(G)=\max\{|I|:\, I\in \mathcal{ID}(G)\}.$ 
An $\alpha(G)$-\emph{set} is defined as a set $I\in \mathcal{ID}(G)$ of cardinality $|I|=\alpha(G)$.
A \emph{clique} is a complete subgraph of $G$. In particular, we consider a single vertex to form a clique of order one. Observe that every independent set of $G$ induces
a clique in $\overline{G}$ and vice versa. Let $\mathcal{C}l(G)$ denote the family of subsets of $V(G)$ that induce a clique in $G$. The \emph{clique number} of $G$, denoted by $\omega(G)$, is the maximum order among all cliques in $G$, that is, $\omega(G)=\max\{|J|:\, J\in \mathcal{C}l(G)\}$. An $\omega(G)$-\emph{set} is defined as a set $J\in \mathcal{C}l(G)$ of cardinality $|J|=\omega(G)$. A clique of order three is called a \emph{triangle} and a graph is \emph{triangle-free} if it contains no triangle. 

\vspace{.1cm}

\noindent
We denote by $K_n$, $N_n$, $W_n$, $P_n$, $C_n$ and $K_{r,n-r}$ (with $r\in \{1, \ldots ,\left\lfloor n/2\right\rfloor\}$) the complete graph, empty graph, wheel graph, path graph, cycle graph and complete bipartite graph of order $n$, respectively. A double star is a tree of diameter exactly three, and we let $\mathcal{DS}_n$ denote the family of all double stars of order $n$. Given a graph $G$, the graph $K_1+G$ is the graph with vertex set $V(K_1+G)=V(K_1)\cup V(G)$ and edge set $E(K_1+G)=E(G)\cup \{v_iv_j: v_i\in V(K_1), v_j\in V(G)\}$.

\section{Results}

We begin this section providing a closed formula for the independent domination number of $\mathtt{C}(G)$. 
	
\begin{theorem}\label{theo-i-C(G)}
Let $G$ be a connected graph of order $n\geq 3$. Then
$$i (\mathtt{C}(G)) = \min \{ |X| + |E(G-X)| + \delta(P_{G}(X)): X\in \mathcal{C}l(G)\}.$$
\end{theorem}
	
\begin{proof}
Let $S\in \mathcal{C}l(G)$ be a set for which $|S| + |E(G-S)| + \delta(P_{G}(X))$ is minimum among all sets in $\mathcal{C}l(G)$. If $P_{G}(S)\neq \emptyset$, assume without loss of generality that $v_1\in P_{G}(S)$. We now define a set $S'\subseteq V(\mathtt{C}(G))$ as follows.
$$S' = 
\begin{cases}
S \cup V_E(G-S) \cup \{v_1\} & \text{if } P_{G}(S) \neq \emptyset, \\[4pt]
S \cup V_E(G-S) & \text{otherwise.}
\end{cases}$$
It is straightforward to verify that $S'\in \mathcal{ID}(\mathtt{C}(G))$. Hence, 
\begin{equation}\label{eq-1}
\begin{split}
i(\mathtt{C}(G)) \leq |S'|&=|S|+ |V_E(G-S)| + \delta(P_{G}(X)) \\[4pt]
&= \min \{ |X| + |E(G-X)| + \delta(P_{G}(X)): X\in \mathcal{C}l(G)\}.
\end{split}
\end{equation}

\noindent
Now, let $D$ be an $i(\mathtt{C}(G))$-set such that $|D\cap V(G)|$ is maximum among all $i(\mathtt{C}(G))$-sets, and let $D'=D\cap V(G)$. Suppose that $D'=\emptyset$. This implies that $D=V_E(G)$. Let $B$ be an $\omega(G)$-set and define $B'=B\cup V_E(G-B)$. It is straightforward to verify that $B'\in \mathcal{ID}(\mathtt{C}(G))$. Hence, 
\begin{equation*}
i(\mathtt{C}(G))\leq |B'|=|B|+|V_E(G-B)|\leq |E(G)|=|V_E(G)|=i(\mathtt{C}(G)).
\end{equation*}

\noindent
Thus, $B'$ is an $i(\mathtt{C}(G))$-set. Moreover, $|B'\cap V(G)|=|B|=\omega(G)>0=|D'|$, which contradicts the maximality of $|D'|$. Therefore, $D'\neq \emptyset$.
Now, observe that $D'\in \mathcal{C}l(G)$, $|D\cap V_E(G)| = |E(G-D')|$ and $P_{G}(D')=\emptyset$. Hence, 
\begin{equation}\label{eq-2}
\begin{split}
i(\mathtt{C}(G)) &= |D\cap V(G)| + |D\cap V_E(G)| \\[3pt]
&= |D'| + |E(G-D')| + \delta(P_{G}(D')) \\[3pt]
&\geq \min \{ |X| + |E(G-X)| + \delta(P_{G}(X)): X\in \mathcal{C}l(G)\}.
\end{split}
\end{equation}
From \eqref{eq-1} and \eqref{eq-2} we obtain the next desired equality.
$$i (\mathtt{C}(G)) = \min \{ |X| + |E(G-X)| + \delta(P_{G}(X)): X\in \mathcal{C}l(G)\}.$$
Therefore, the proof is complete.
\end{proof}

\noindent
Determining the value of $i (\mathtt{C}(G))$ directly from its definition can be difficult in general. However, the formula established in Theorem~\ref{theo-i-C(G)} allows us to derive bounds and, in several cases, exact values for this parameter. The next results illustrate some applications of this formula.

\begin{theorem}\label{theo-triangle-free}
Let $G$ be a connected triangle-free graph of size $m\geq 2$. Then
$$i(\mathtt{C}(G)) = m+3 - \max \{|N_G(v_i)| + |N_G(v_j)|:\ v_iv_j \in E(G)\}.$$
\end{theorem}
	
\begin{proof}
If $G \cong K_{1,n-1}$, then it is straightforward to verify that 
\begin{equation*}
i(\mathtt{C}(K_{1,n-1})) = 2 = m + 3  - \max \{|N_{K_{1,n-1}}(v_i)| + |N_{K_{1,n-1}}(v_j)|:\ v_iv_j \in E(K_{1,n-1})\}.
\end{equation*}
From now on assume that $G \ncong K_{1,n-1}$. Let $v_kv_l\in E(G)$ be an edge which satisfies that $|N_G(v_k)|+ |N_G (v_l)| = \max \{|N_G(v_i)| + |N_G(v_j)|:\ v_iv_j \in E(G)\}$ and let $S = \{v_k,v_l\}$. Since $G$ is triangle-free, it follows that $\delta(P_{G}(S)) = 0$. Moreover, we have that $|E(G-S)| = m - (|N_G(v_k)| + |N_G (v_l)|) + 1$. Since $S\in \mathcal{C}l(G)$, it follows from Theorem~\ref{theo-i-C(G)} that
\begin{equation*}
\begin{split}
i(\mathtt{C}(G)) &= \min \{ |X| + |E(G-X)| + \delta(P_{G}(X)) :\ X\in \mathcal{C}l(G)\} \\[4pt]
&\leq |S| + |E(G-S)| + \delta(P_{G}(S)) \\[4pt]
&= 2+m-   (|N_G(v_k)| + |N_G (v_l)|)+1 \\[4pt]
&= m+3 - \max \{|N_G(v_i)| + |N_G(v_j)|:\ v_iv_j \in E(G)\}.
\end{split}
\end{equation*}
We only need to prove that $i(\mathtt{C}(G)) \geq m+3 - \max \{|N_G(v_i)| + |N_G(v_j)|:\ v_iv_j \in E(G)\}$. By Theorem~\ref{theo-i-C(G)} there exists $S' \in \mathcal{C}l(G)$ such that $i(\mathtt{C}(G)) = |S'| + |E(G-S')| + \delta(P_{G}(S'))$. Since $G$ is triangle-free, we have that $|S'| \leq 2$. We now consider the following two complementary cases.

\vspace{.2cm}

\noindent
Case 1: $|S'| = 1$. In this case, the minimality of $i(\mathtt{C}(G))$ implies that $|E(G-S')| = m - \Delta(G)$. Moreover, since $G\not\cong K_{1,n-1}$, it follows that $\Delta(G)\leq \max \{|N_G(v_i)| + |N_G(v_j)|:\ v_iv_j \in E(G)\}-2$. Hence, 
\begin{equation*}
\begin{split} 
i(\mathtt{C}(G)) &= |S'| + |E(G-S')| + \delta(P_{G}(S')) \\[4pt]
&\geq 1 + m-\Delta(G) \\[4pt]
&\geq m + 3 - \max \{|N_G(v_i)| + |N_G(v_j)|:\ v_iv_j \in E(G)\}.
\end{split}
\end{equation*}

\vspace{.2cm}	

\noindent
Case 2: $|S'| =2$. In this case, the minimality of $i(\mathtt{C}(G))$ implies that $|E(G-S')| = m - \max \{|N_G(v_i)| + |N_G(v_j)|:\ v_iv_j \in E(G)\} + 1$. Since $G$ is triangle-free, it follows that  $\delta(P_{G}(S')) = 0$. Hence,
\begin{equation*}
\begin{split}
i(\mathtt{C}(G)) & = |S'| + |E(G-S')| + \delta(P_{G}(S')) \\[4pt]
& = m+3 - \max \{|N_G(v_i)| + |N_G(v_j)|:\ v_iv_j \in E(G)\}.
\end{split}
\end{equation*}

\vspace{.2cm}

\noindent
From the previous two cases, we obtain that $i(\mathtt{C}(G)) \geq m+3 - \max \{|N_G(v_i)| + |N_G(v_j)|:\ v_iv_j \in E(G)\}$, which completes the proof.
\end{proof}

\noindent
An immediate consequence of the preceding theorem is the following corollary, which provides the exact value for $i(\mathtt{C}(G))$ when $G$ is a path, a cycle, or a complete bipartite graph.

\begin{corollary}\label{cor-Cn}
The following equalities hold for any integer $n\geq 4$.
\begin{enumerate}[label={\rm(\roman*)}]
%\item \label{cor-1} $i(\mathtt{C}(P_n)) = n-2$.
\item \label{cor-2} $i(\mathtt{C}(C_n)) = i(\mathtt{C}(P_n))+1=n-1$.
\item \label{cor-3} $i(\mathtt{C}(K_{r,n-r})) = r(n-r)-n+3$.
\end{enumerate}
\end{corollary}

\noindent
The following result provides tight lower and upper bounds for the independent domination number of the 
central graph $\mathtt{C}(G)$ in terms of the order and size of $G$.

\begin{theorem}
Let $G$ be a connected graph of order $n\geq 3$ and size $m$. Then the following statements hold.
\begin{enumerate}[label={\rm(\roman*)}]
\item \label{theo-1} $1+\left\lceil m/n\right\rceil\leq i(\mathtt{C}(G))\leq m$.
\item \label{theo-2} $i(\mathtt{C}(G))=2$ if and only if $G\in \{C_3,K_{1,n-1}\}\cup \mathcal{DS}_n$.
\item \label{theo-3} $i(\mathtt{C}(G)) =m$ if and only if $G\cong P_3$.
\end{enumerate}
\end{theorem}

\begin{proof}
Observe that $V_E(G)\in \mathcal{ID}(\mathtt{C}(G))$. Hence, $i(\mathtt{C}(G))\leq |V_E(G)|=m$. In \cite{Berge1973}, Berge proved that $i(H)\geq \left\lceil |V(H)|/(1+\Delta(H))\right\rceil$ for any graph~$H$. Since $|V(\mathtt{C}(G))|=n+m$ and $\Delta(\mathtt{C}(G))=n-1$, it follows that $i(\mathtt{C}(G))\geq \left\lceil (n+m)/n\right\rceil=1+\left\lceil m/n\right\rceil$, which completes the proof of \ref{theo-1}.

\vspace{.1cm}

\noindent
Now, we proceed to prove \ref{theo-2}. It is straightforward to check that $i(\mathtt{C}(G))=2$ for every graph $G\in \{C_3,K_{1,n-1}\}\cup \mathcal{DS}_n$. Conversely, let $G$ be a connected graph satisfying that $i(\mathtt{C}(G))=2$. By Theorem~\ref{theo-i-C(G)}, there exists $S\in \mathcal{C}l(G)$ such that  
$2=i(\mathtt{C}(G))=|S|+|E(G-S)|+\delta(P_{G}(S))$. This implies that $|S|\in \{1,2\}$. We next analyze the following cases.

\vspace{.1cm}

\noindent
Case 1: $|S|=1$. In this case, it follows that $|E(G-S)|\leq 1$. If $|E(G-S)|=1$, then $|V(G-S)|=2$, which implies that $n=3$, and hence $G\in \{P_3,C_3\}$. From now on, assume that $E(G-S)=\emptyset$. Thus, $G-S\cong N_{n-1}$. Since $G$ is connected, it follows that $G\cong K_1+N_{n-1}\cong K_{1,n-1}$.

\vspace{.1cm}

\noindent
Case 2: $|S|=2$. In this case, it follows that $E(G-S)=\emptyset$. Thus, $G-S\cong N_{n-2}$. In addition, we have that $\delta(P_{G}(S))=0$. This implies that $|N_G(v_i)\cap S|=1$ for every vertex $v_i\in V(G)\setminus S$. As a consequence, $G\in \{K_{1,n-1}\}\cup \mathcal{DS}_n$.

\vspace{.1cm}

\noindent
From the previous two cases we conclude that $G\in \{C_3,K_{1,n-1}\}\cup \mathcal{DS}_n$, which completes the proof of \ref{theo-2}. 

\vspace{.1cm}

\noindent
Finally, we proceed to prove \ref{theo-3}. It is straightforward to check that $i(\mathtt{C}(P_3))=2=|E(P_3)|$. Conversely, let $G$ be a connected graph satisfying that $i(\mathtt{C}(G))=m$. Let $v_k\in V(G)$ be a vertex such that $|N_G(v_k)|=\Delta(G)$. Since $\{v_k\}\in \mathcal{C}l(G)$, it follows by Theorem~\ref{theo-i-C(G)} that
$m=i(\mathtt{C}(G))\leq |\{v_k\}|+|E(G-\{v_k\})|+\delta(P_{G}(\{v_k\}))\leq 2+m-\Delta(G)$. Hence, $\Delta(G)=2$, which implies that $G\in \{P_n,C_n\}$. If $n\geq 4$, then by Corollary~\ref{cor-Cn}~\ref{cor-2} it follows that $i(\mathtt{C}(G))\neq m$, a contradiction. Thus, $n=3$. Moreover, \ref{theo-2} leads to $G\not\cong C_3$. Therefore, $G\cong P_3$, which completes the proof of \ref{theo-3}.
\end{proof}

\noindent
The following theorem provides lower and upper bounds for the independent domination number of $\mathtt{C}(G))$ in terms of the size, clique number, maximum degree and minimum degree of $G$. We first establish the following useful lemma.

\begin{lemma}\label{lem-f}
Let $a\geq 2$ be a positive integer and let $f: [1,a] \to \mathbb{R}$ be the function defined by $f(x) = \frac{x^2-x(2a-1)}{2}$. Then the following statements hold.
\begin{enumerate}[label={\rm(\roman*)}]
\item \label{lem-f1} $f(a) = f(a-1)$.
\item \label{lem-f2} $f$ is strictly decreasing on $[1,a-1]$.
\end{enumerate}
\end{lemma}

\begin{proof}
It is straightforward to verify that \ref{lem-f1} holds. We now proceed to prove \ref{lem-f2}. Observe that $f'(x) = x-a+1/2$. Let $x\in[1, a-1]$. Since $x\leq a-1 < a - 1/2$, it follows that $f'(x)<0$. Hence, $f$ is strictly decreasing on $[1,a-1]$, which completes the proof.
\end{proof}

\begin{theorem}\label{theo-lower-upper-bounds}
Let $G$ be a connected graph of size $m\geq 2$ which is not a complete graph. Then
$$m- \frac{\omega(G)(2\Delta(G)-\omega(G)-1)}{2}\leq i(\mathtt{C}(G))\leq m- \frac{\omega(G)(2\delta(G)-\omega(G)-1)}{2}.$$
\end{theorem}

\begin{proof}
Let $B$ be an $\omega(G)$-set. The following equalities follow from a counting argument on the number of edges in $G-B$.
\begin{equation}\label{eq-33-upper}
\begin{split}
|E(G-B)| &= m - |E_G(B,B)| - |E_G(B, V(G) \setminus B)| \\[4pt]
&= m - \sum_{v \in B}|N_G(v)|+\frac{\omega(G)(\omega(G)-1)}{2}.
\end{split}
\end{equation}
By definition, $\delta(P_{G}(B)) = 0$. Moreover, since $\sum_{v \in B} |N_G(v)| \geq \omega(G)\delta(G)$, it follows from Theorem~\ref{theo-i-C(G)} and \eqref{eq-33-upper} that		
\begin{equation*}\label{eq-4-upper}
\begin{split}
i(\mathtt{C}(G))&\leq |B| + |E(G-B)| + \delta(P_{G}(B))\\[4pt] 
&=  \omega(G) + m - \sum_{v \in B} |N_G(v)| + \omega(G)(\omega(G)-1)/2  \\[4pt] 
&\leq \omega(G) + m - \omega(G)\delta(G) +\omega(G)(\omega(G)-1)/2 \\[5pt] 
&= m - \frac{\omega(G) ( 2\delta(G) - \omega(G) - 1)}{2},
\end{split}
\end{equation*}
which completes the proof of the upper bound. Finally, we proceed to prove the lower bound. By Theorem~\ref{theo-i-C(G)}, there exists a set $S \in \mathcal{C}l(G)$ such that 
$i(\mathtt{C}(G))=|S|+ |E(G-S)| + \delta(P_{G}(S))$.
Now, let $s = |S|$.  The following equalities follow from a counting argument on the number of edges in $G-S$.
\begin{equation}\label{eq-33}
\begin{split}
|E(G-S)| &= m - |E_G(S,S)| - |E_G(S, V(G) \setminus S)| \\[4pt]
&= m - \sum_{v \in S} |N_G(v)| + \frac{s(s-1)}{2}. 
\end{split}
\end{equation}

\noindent
From \eqref{eq-33} and the fact that $\sum_{v \in S} |N_G(v)| \leq s\Delta(G)$, it follows that 
\begin{equation*}\label{eq-4}
\begin{split}
i(\mathtt{C}(G))&=|S| + |E(G-S)| + \delta(P_{G}(S))\\[4pt] 
&\geq  s + m - \sum_{v \in S} |N_G(v)| + s(s-1)/2  \\[4pt] 
&\geq s + m - s\Delta(G) +s(s-1)/2 \\[7pt] 
&= m +(s^2-s(2\Delta(G)-1))/2.
\end{split}
\end{equation*}

\vspace{.15cm}

\noindent
In particular, we have that $i(\mathtt{C}(G))\geq m+f(s)$, where $f(s) =(s^2-s(2\Delta(G)-1))/2$. Since $G\not\cong K_n$, it follows that $1\leq s\leq \omega(G)\leq \Delta(G)$. Moreover, by Lemma \ref{lem-f}~\ref{lem-f2}, we have that $f$ is strictly decreasing on $[1,\Delta(G)-1]$. This implies that if $\omega(G)\leq \Delta(G)-1$, then $f(s)$ attains its minimum value at $s=\omega(G)$. Therefore, 
$$i(\mathtt{C}(G))\geq m+f(s)\geq m +f(\omega(G)) = m-\frac{\omega(G)(2\Delta(G)- \omega(G) - 1) )}{2},$$

\noindent
as desired. Finally, let us consider that $\omega(G)=\Delta(G)$. By Lemma \ref{lem-f}~\ref{lem-f1}, it follows that $f(\Delta(G)-1)=f(\Delta(G))=f(\omega(G))$. Applying again the fact that $f$ is strictly decreasing on $[1,\Delta(G)-1]$, it follows that $f(\Delta(G)-1)=\min\{f(1),\ldots, f(\Delta(G)-1)\}$. Then $f(s)$ attains its minimum value at $s=\Delta(G)-1$ or $s=\Delta(G)$. Therefore,
\begin{equation*}
\begin{split}
i(\mathtt{C}(G)) \geq m+\min\left\{f(\Delta(G)-1),f(\Delta(G))\right\}&= m +f(\omega(G))\\[5pt] 
&= m - \frac{\omega(G) (2\Delta(G) - \omega(G) - 1)}{2},
\end{split}
\end{equation*}
as desired. Therefore, the proof is complete.
\end{proof}

\noindent
An immediate consequence of the preceding theorem is the following corollary, which gives a closed formula for $i(\mathtt{C}(G))$ when $G$ is a $k$-regular graph.

\begin{corollary}
Let $G$ be a connected $k$-regular graph of order $n\geq 3$, distinct from a complete graph. Then
$$i(\mathtt{C}(G)) = \frac{nk+ \omega(G)^{2} - (2k-1) \omega(G)}{2}.$$
\end{corollary}

\vspace{.2cm}

\noindent	
Next, we establish an interesting relationship between the independent domination numbers of $\mathtt{C}(G)$ and $\mathtt{C}(K_1+G)$. In particular, we prove that $i(\mathtt{C}(K_1+G)) = i(\mathtt{C}(G))+1.$ The following lemma will be useful.

\begin{lemma}\label{lem:K1G}
Let $G$ be a connected graph of order $n\geq 3$. Then there exists a set $S \in \mathcal{C}l(K_1+G)$ which satisfies the following conditions.
\begin{enumerate}[label={\rm(\roman*)}]
\item \label{lem-k11} $S\not\subseteq V(G)$.
\item \label{lem-k12} $i(\mathtt{C}(K_1+G)) = |S| + |E((K_1+G)-S)| + \delta(P_{K_1+G}(S))$.
\end{enumerate}
\end{lemma}

\begin{proof}
Let $V(K_1+G)=\{v_1,\ldots,v_n,v_{n+1}\}$ such that $V(G)=\{v_1,\ldots,v_n\}$ and $V(K_1) = \{v_{n+1}\}$. 
By Theorem~\ref{theo-i-C(G)}, there exists a set $S \in \mathcal{C}l(K_1+G)$ such that 
\begin{equation*}
i(\mathtt{C}(K_1+G)) = |S| + |E((K_1+G)-S)| + \delta(P_{K_1+G}(S)).
\end{equation*}
Hence, \ref{lem-k12} holds. Without loss of generality, assume that $|S\cap \{v_{n+1}\}|$ is maximum among all sets $Y \in \mathcal{C}l(K_1+G)$ with $i(\mathtt{C}(K_1+G)) = |Y| + |E((K_1+G)-Y)| + \delta(P_{K_1+G}(Y))$. Suppose that $v_{n+1} \notin S$.
Next, let us consider the following three complementary cases.

\vspace{.2cm}

\noindent
Case 1: $\delta(P_{K_1+G}(S))=1$. In this case, it follows that $S=V(G)$ and $P_{K_1+G}(S)=\{v_{n+1}\}$. Therefore, $K_1+G\cong K_{n+1}$.  Let $S'=(S\setminus \{v_1\}) \cup \{v_{n+1}\}$. Observe that $|S'|=|S|$, $|E((K_1+G)-S')|=|E((K_1+G)-S)|$ and $\delta(P_{K_1+G}(S'))=1$. Hence, 
\begin{equation}\label{eq-lem-K1}
\begin{split}
i(\mathtt{C}(K_1+G)) &= |S| + |E((K_1+G)-S)| + \delta(P_{K_1+G}(S))\\[4pt]
&=|S'| + |E((K_1+G)-S')| + \delta(P_{K_1+G}(S')).
\end{split}
\end{equation}
This is a contradiction, since $S'\in \mathcal{C}l(K_1+G)$ and $|S'\cap \{v_{n+1}\}| > |S\cap \{v_{n+1}\}|$.  

\vspace{.25cm}

\noindent
Case 2: $\delta(P_{K_1+G}(S))=0$ and $|S| = n-1$. Without loss of generality, suppose that $V(G)\setminus S=\{v_1\}$. Since $G$ is a connected graph, there exists a vertex $v_k \in S$ such that $v_1v_k \in E(G)$. Let $S'=(S\setminus \{v_k\}) \cup \{v_{n+1}\}$. Observe that $|S'|=|S|$, $|E((K_1+G)-S')|=|E((K_1+G)-S)|$ and $\delta(P_{K_1+G}(S'))=0$. Hence, the equality in \eqref{eq-lem-K1} holds. This is a contradiction, since $S'\in \mathcal{C}l(K_1+G)$ and $|S'\cap \{v_{n+1}\}| > |S\cap \{v_{n+1}\}|$.

\vspace{.25cm}

\noindent
Case 3: $\delta(P_{K_1+G}(S))=0$ and $|S|\leq n-2$. Let $S' = S\cup \{v_{n+1}\}$. Observe that $|S|=|S'|-1$ and $|E((K_1+ G) - S)| = |E((K_1+ G) - S')|+ n-|S|$. Hence,
\begin{equation*}
\begin{split}
i(\mathtt{C}(K_1+ G))&= |S| + |E((K_1+G)-S)| + \delta(P_{K_1+G}(S))\\[4pt]
&=|S'|-1 + |E((K_1+ G) - S')|+n-|S| \\[4pt]
&\geq |S'| + |E((K_1+G)-S')| + \delta(P_{K_1+G}(S')).
\end{split}
\end{equation*}
Since $S'\in \mathcal{C}l(K_1+G)$, it follows from Theorem~\ref{theo-i-C(G)} that $i(\mathtt{C}(K_1+ G))=|S'| + |E((K_1+G)-S')| + \delta(P_{K_1+G}(S'))$. This is a contradiction, since $|S'\cap \{v_{n+1}\}| > |S\cap \{v_{n+1}\}|$.

\vspace{.25cm}

\noindent
As shown in the previous cases, we obtain a contradiction. Therefore, $v_{n+1} \in S$, that is, $S\not\subseteq V(G)$, which completes the proof.
\end{proof}

\begin{theorem}\label{theo:iCK1+G}
Let $G$ be any connected graph of order $n \geq 3$. Then
$$i(\mathtt{C}(K_1+G)) = i(\mathtt{C}(G))+1.$$
\end{theorem}
	
\begin{proof}
Let $V(K_1+G)=\{v_1,\ldots,v_n,v_{n+1}\}$ such that $V(G)=\{v_1,\ldots,v_n\}$ and $V(K_1) = \{v_{n+1}\}$. 
By Theorem~\ref{theo-i-C(G)}, there exists $Y\in\mathcal{C}l(G)$ such that $i (\mathtt{C}(G)) = |Y|+|E(G-Y)|+\delta(P_{G}(Y))$. Let $Y' = Y\cup \{v_{n+1}\}$. Note that $|Y'| = |Y| + 1$, $|E((K_1+G)-Y')| = |E(G-Y)|$ and $\delta(P_{K_1 +G}(Y')) = \delta(P_{G}(Y))$. Since $Y' \in \mathcal{C}l (K_1 +G)$, it follows from Theorem~\ref{theo-i-C(G)} that
\begin{equation*}
\begin{split}
i(\mathtt{C}(K_1+G)) 
&=\min\{|X|+|E((K_1+G)-X)|+\delta(P_{K_1 + G}(X)): X \in\mathcal{C}l(K_1+G)\} \\[4pt]
&\leq |Y'|+|E((K_1+G)-Y')|+\delta(P_{K_1 + G}(Y')) \\[4pt]
&= |Y|+1+|E(G-Y)|+\delta(P_{G}(Y))\\[4pt]
& = i(\mathtt{C}(G))+1.
\end{split}
\end{equation*}
We only need to prove that $i(\mathtt{C}(G)) + 1\leq i(\mathtt{C}(K_1+ G))$. By Lemma~\ref{lem:K1G}, there exists $Z \in \mathcal{C}l(K_1+G)$ such that $i(\mathtt{C}(K_1+G))  = |Z| + |E((K_1+G) - Z)| + \delta(P_{K_1+G}(Z))$ and $v_{n+1}\in Z$. Let $Z' = Z\setminus \{v_{n+1}\}$. 
Note that $|Z'|=|Z|-1$, $|E(G-Z')| = |E((K_1+ G) - Z)|$ and $\delta(P_{G}(Z'))=\delta(P_{K_1 + G}(Z))$. Since $Z' \in \mathcal{C}l(G)$, it follows from Theorem~\ref{theo-i-C(G)} that
\begin{equation*}
\begin{split}
i(\mathtt{C}(G))+1 
&=\min\{|X|+|E(G-X)|+\delta(P_{G}(X)): X \in\mathcal{C}l(G)\}+1 \\[4pt]
&\leq |Z'|+|E(G-Z')|+\delta(P_{G}(Z'))+1 \\[4pt]
&= |Z|+|E((K_1+G)-Z)|+\delta(P_{K_1+ G}(Z))\\[4pt]
& = i(\mathtt{C}(K_1+G)).
\end{split}
\end{equation*}
Therefore, $i(\mathtt{C}(K_1+ G))=i(\mathtt{C}(G)) + 1$, which completes the proof.
\end{proof}
	
\noindent
The previous theorem allows us to compute the exact value of $i(\mathtt{C}(G))$ when $G$ is a wheel or a complete graph.

\begin{proposition}
Let $n\geq 4$ be an integer. Then
$$i(\mathtt{C}(K_n)) =i(\mathtt{C}(W_n))= n-1.$$
\end{proposition}

\begin{proof}
It is easy to check that $i(\mathtt{C}(K_3))=2$. For every $r\geq 4$, we have that  $K_r\cong K_1+K_{r-1}$. By applying Theorem~\ref{theo:iCK1+G} consecutively $n-3$ times, we obtain that  $i(\mathtt{C}(K_n))= n-1$, as desired. Now, observe that $W_n\cong K_1+C_{n-1}$. By Theorem~\ref{theo:iCK1+G} and Corollary~\ref{cor-Cn}~\ref{cor-2} it follows that $i(\mathtt{C}(W_n))=i(\mathtt{C}(K_1+C_{n-1}))=i(\mathtt{C}(C_{n-1}))+1 =n-1.$ Therefore, the proof is complete.
\end{proof}

\noindent
We conclude this section by providing a closed formula for the independent domination number of $\mathtt{C}^2(G)$. We first establish the following useful result.

\begin{proposition}\label{prop-clique}
For any nontrivial connected graph $G$ which is not a complete graph,
$$\omega(\mathtt{C}(G)) = \alpha(G).$$
\end{proposition}

\begin{proof}
First, we proceed to prove that $\omega(\mathtt{C}(G)) \leq \alpha(G)$. Let $S$ be an $\omega(\mathtt{C}(G))$-set such that $|S \cap V(G)|$ is maximum. Suppose that $S \cap V_E(G) \neq \emptyset$. Then $|S \cap V_E(G)|=1$, which implies that $|S| = 2$. Since $G$ is not a complete graph, there exist two distinct vertices $v_i,v_j \in V(G)$ such that $v_iv_j \notin E(G)$. Hence, $\{v_i,v_j\}$ is an $\omega(\mathtt{C}(G))$-set, and $|\{v_i,v_j\}\cap V(G)| > |S\cap V(G)|$, a contradiction. Therefore, $S \subseteq V(G)$, and thus $S$ is an independent set of $G$. Consequently, $\omega(\mathtt{C}(G)) = |S| \leq \alpha(G)$. Finally, we proceed to prove that $\alpha(G) \leq \omega(\mathtt{C}(G))$. Let $I$ be an $\alpha(G)$-set. It is straightforward to verify that $I \in \mathcal{C}l(\mathtt{C}(G))$. Therefore, $\alpha(G) = |I| \leq \omega(\mathtt{C}(G))$, which completes the proof.
\end{proof}

\begin{theorem}
Let $G$ be a connected graph of order $n\geq 3$ and size $m$. Then
$$i (\mathtt{C}^{2}(G)) = m + \frac{n(n-1)}{2} + \frac{\alpha(G)^{2} - \alpha(G) (2n-3)}{2}.$$
\end{theorem}

\begin{proof}
First, we observe that 
\begin{equation}\label{eq-C2-A1}
|E(\mathtt{C}(G))| = 2|E(G)| + |E(\overline{G})| = m + \frac{n(n-1)}{2}.
\end{equation}
Let $I$ be an $\alpha(G)$-set. It is straightforward to verify that $I \in \mathcal{C}l (\mathtt{C}(G))$, and by Proposition~\ref{prop-clique} we have that $|I|=\omega(\mathtt{C}(G))$. Consequently, $I$ is also an $\omega(\mathtt{C}(G))$-set.  The following equalities follow from a counting argument on the number of edges in $\mathtt{C}(G)-I$, together with the fact that $\sum_{v\in I} |N_{\mathtt{C}(G)} (v)| = \alpha(G) (n-1)$ and \eqref{eq-C2-A1}.
\begin{equation}\label{eq-C2-A2}
\begin{split}
|E(\mathtt{C}(G)-I)| &= |E(\mathtt{C}(G))| -  |E_{\mathtt{C}(G)} (I,I)| - |E_{\mathtt{C}(G)} (I, V(\mathtt{C}(G)) \setminus I)| \\[5pt]
&= m + \frac{n(n-1)}{2} - \sum_{v\in I} |N_{\mathtt{C}(G)} (v)| + \frac{\alpha(G)(\alpha(G)-1)}{2} \\[5pt]
&= m + \frac{n(n-1)}{2} + \frac{\alpha(G)^{2} -\alpha(G) (2n-1)}{2}.
\end{split}
\end{equation}
Since $I$ is an $\omega(\mathtt{C}(G))$-set, it follows that $\delta(P_{\mathtt{C}(G)}(I)) = 0$. From Theorem~\ref{theo-i-C(G)} and \eqref{eq-C2-A2} we obtain that
\begin{equation}\label{eq-C2-up}
\begin{split}
i (\mathtt{C}^{2}(G)) &\leq |I| + |E(\mathtt{C}(G)-I)| + \delta(P_{\mathtt{C}(G)}(I)) \\[5pt]
&= \alpha (G) + m + \frac{n(n-1)}{2} + \frac{\alpha(G)^{2} -\alpha(G) (2n-1)}{2}\\[5pt]
&= m + \frac{n(n-1)}{2} + \frac{\alpha(G)^{2} -\alpha(G) (2n-3)}{2}.
\end{split}
\end{equation}

\vspace{.2cm}

\noindent
Finally, we prove that \eqref{eq-C2-up} consists of a chain of equalities. To this end, we consider the following two cases.

\vspace{.2cm}

\noindent
Case 1: $G\not\cong K_n$. By Proposition \ref{prop-clique} we have that $\omega (\mathtt{C}(G)) = \alpha(G)$. By the lower bound given in Theorem~\ref{theo-lower-upper-bounds}, the fact that $\Delta(\mathtt{C}(G)) = n-1$ and \eqref{eq-C2-A1} we obtain that
\begin{equation}\label{eq-C2-low}
\begin{split}
i (\mathtt{C}^{2}(G)) &\geq |E(\mathtt{C}(G))| - \frac{\omega(\mathtt{C}(G))(2\Delta(\mathtt{C}(G))- \omega(\mathtt{C}(G)) -1 )} {2} \\[6pt]
&= m + \frac{n(n-1)}{2} + \frac{\alpha(G)^{2} - \alpha(G) (2n-3)}{2}.
\end{split}
\end{equation}
Therefore, from \eqref{eq-C2-up} and \eqref{eq-C2-low} it follows that
$$i (\mathtt{C}^{2}(G)) = m + \frac{n(n-1)}{2} + \frac{\alpha(G)^{2} - \alpha(G) (2n-3)}{2}.$$

\vspace{.2cm}

\noindent
Case 2: $G \cong K_n$. By Theorem~\ref{theo-i-C(G)}, there exists a set $S \in \mathcal{C}l (\mathtt{C}(G))$ such that 
$i(\mathtt{C}^{2}(G)) = |S| + |E(\mathtt{C}(G)-S)| + \delta(P_{\mathtt{C}(G)}(S))$. It is straightforward to verify that $|S|\leq 2$. Since $n\geq 3$, it follows that $\delta(P_{\mathtt{C}(G)}(S)) = 0$. Hence,
\begin{equation}\label{eq-C2-A3}
i(\mathtt{C}^{2}(G)) = |S| + |E(\mathtt{C}(G)-S)|.
\end{equation}
Next, we consider the following three complementary subcases.

\vspace{.2cm}
		
\noindent
Subcase 2.1: $|S| = 1$ and $S \subseteq V_E(G)$. Observe that $|E(\mathtt{C}(G)-S)| = n(n-1) - 2$. Substituting this equality into \eqref{eq-C2-A3}, we obtain that $i(\mathtt{C}^{2}(G))= n(n-1)-1$.
If $n\geq 4$, then we obtain a contradiction to \eqref{eq-C2-up}. Hence, $n=3$ and consequently, $$i(\mathtt{C}^{2}(G)) = 5 = m + \frac{n(n-1)}{2} + \frac{\alpha(G)^{2} -\alpha(G) (2n-3)}{2}.$$ 

\vspace{.2cm}
	
\noindent
Subcase 2.2: $|S| = 1$ and $S \subseteq V(G)$. Observe that $|E(\mathtt{C}(G)-S)| = (n-1)^{2}$. Substituting this equality into \eqref{eq-C2-A3}, we obtain that
\begin{equation*}
i(\mathtt{C}^{2}(G)) = 1 + (n-1)^{2}=m + \frac{n(n-1)}{2} + \frac{\alpha(G)^{2} -\alpha(G) (2n-3)}{2}.
\end{equation*}

\vspace{.2cm}

\noindent
Subcase 2.3: $|S| = 2$. Observe that $|E(\mathtt{C}(G)-S)| = (n-1)^{2} - 1$. Substituting this equality into \eqref{eq-C2-A3}, we obtain that
\begin{equation*}
i(\mathtt{C}^{2}(G)) = 1 + (n-1)^{2}=   m + \frac{n(n-1)}{2} + \frac{\alpha(G)^{2} -\alpha(G) (2n-3)}{2}
\end{equation*}

\vspace{.25cm}

\noindent
From the previous two cases, we conclude that $i(\mathtt{C}^{2}(G)) = m + \frac{n(n-1)}{2} + \frac{\alpha(G)^{2} -\alpha(G) (2n-3)}{2}$, which completes the proof.
\end{proof}

\section*{Declarations}

\noindent
\textbf{Data availability}: Not applicable.

\vspace{.2cm}

\noindent
\textbf{Conflict of interest}: The authors declare no conflict of interest.

\end{document}